\documentclass[11pt]{article}
\usepackage{times}
\usepackage[english]{babel}
\usepackage[dvips]{graphicx}
\usepackage{amssymb}
\usepackage{amsfonts}
\usepackage{amsmath}
\usepackage{mathrsfs}
\usepackage{url} 
\usepackage{setspace}
\usepackage[numbers,sort&compress,longnamesfirst,sectionbib]{natbib}             

\graphicspath{{./figs/}}

\newtheorem{rem}{Remark}

\newcommand{\qed}{\hfill \mbox{\raggedright $\Box$}}

\newcommand{\Th}{\mathacal{\Theta}}
\newcommand{\Oh}{\mathcal{O}}

\DeclareSymbolFont{AMSb}{U}{msb}{m}{n}

\newcommand{\Z}{{\mathbb{Z}}}
\newcommand{\R}{{\mathbb{R}}}

\renewcommand{\leq}{\leqslant}
\renewcommand{\geq}{\geqslant}

\newtheorem{theorem}{Theorem}

\newtheorem{lemma}[theorem]{Lemma}

\newtheorem{claim}[theorem]{Claim}

\newenvironment{proof}{\noindent{\bf Proof}\hspace*{1em}}{\qed\bigskip}
\newenvironment{proof-sketch}{\noindent{\bf Sketch of Proof}\hspace*{1em}}{\qed\bigskip}
\newenvironment{proof-idea}{\noindent{\bf Proof Idea}\hspace*{1em}}{\qed\bigskip}
\newenvironment{proof-of-lemma}[1]{\noindent{\bf Proof of Lemma #1}\hspace*{1em}}{\qed\bigskip}
\newenvironment{proof-attempt}{\noindent{\bf Proof Attempt}\hspace*{1em}}{\qed\bigskip}

\makeatletter
\def\fnum@figure{{\bf Figure \thefigure}}
\def\fnum@table{{\bf Table \thetable}}
\long\def\@mycaption#1[#2]#3{\addcontentsline{\csname
  ext@#1\endcsname}{#1}{\protect\numberline{\csname
  the#1\endcsname}{\ignorespaces #2}}\par
  \begingroup
    \@parboxrestore
    \small
    \@makecaption{\csname fnum@#1\endcsname}{\ignorespaces #3}\par
  \endgroup}
\def\mycaption{\refstepcounter\@captype \@dblarg{\@mycaption\@captype}}
\makeatother

\newif\ifnotesw
\noteswtrue

\def\pr{\mathbb{P}}
\def\E{\mathbb{E}}

\def\bin{{\rm Bin}}

\def\po{{\rm Po}}

\def\al{\alpha}

\def\gm{\gamma}

\def\lm{\lambda}

\def\th{\theta}
\def\Th{\Theta}

\def\o{o}

\def\beq{\begin{equation}}
\def\eeq{\end{equation}}
\def\beqn{\begin{eqnarray}}
\def\eeqn{\end{eqnarray}}

\title{Bootstrap Percolation on Random Geometric Graphs}
\author{
Milan Bradonji\'c and Iraj Saniee\\
\\
Mathematics of Networks and Communications\\ 
Bell Labs, Alcatel-Lucent\\ 
600 Mountain Avenue, Murray Hill, NJ 07974, USA 
\\
\texttt{\{milan,iis\}@research.bell-labs.com}
}

\date{}

\begin{document}
\maketitle

\begin{abstract}
Bootstrap percolation has been used effectively to model phenomena
as diverse as emergence of magnetism in materials,
spread of infection,
diffusion of software viruses in computer networks,
adoption of new technologies, 
and emergence of collective action and cultural fads
in human societies.
It is defined on an (arbitrary) network of
interacting agents whose state is determined
by the state of their neighbors according to a threshold rule.
In a typical setting, bootstrap percolation starts by random
and independent ``activation'' of nodes with a fixed probability $p$,
followed by
a deterministic process for additional activations based on the
density of active nodes in each neighborhood ($\theta$ activated nodes). 
Here, we study bootstrap percolation on random geometric graphs in
the regime when the latter are (almost surely) connected. Random
geometric graphs provide an appropriate model in settings where the
neighborhood structure of each node is determined by geographical
distance, as in wireless {\it ad hoc} and sensor networks as well as in contagion.
We derive bounds on the critical thresholds $p_c', p_c''$ such that for all $p > p''_c(\theta)$ full percolation takes place, whereas for $p < p'_c(\theta)$ it does not. 
We conclude with simulations that compare 
numerical thresholds with those obtained analytically.
\end{abstract}

\section{Introduction}
\label{sec:intro}

Some crystals or lattices studied in physics and chemistry can be modeled 
as consisting of atoms occupying sites with specified probabilities. The lattice
as a whole would then exhibit certain macroscopic properties, such as (ferro)magnetism,
only when a sufficient number of neighboring sites of each atom are also similarly
occupied.
In computer memory arrays each functional memory unit can be considered as an occupied 
site, and a minimum percentage of functioning units are needed in the
vicinity of each memory unit in order to maintain the array with proper
functioning.
In adoption of new technology or emergence of cultural fads,
an individual is positively influenced 
when a sufficient number of its close friends have also done so. 

All three examples cited above may be modeled via a formal process called
``bootstrap percolation'' which is
a dynamic process that evolves similar to a cellular automaton.
Unlike cellular automata, however, this process can be defined on arbitrary graphs
and starts with random initial conditions.  
Nodes are either active or inactive. Once activated, a node remains active
forever. Each node is initially active with a
(given) probability $p$. Subsequently and at each discrete time step, a node becomes 
active if $\theta$ of its nearest neighbors
are active, for a fixed value of $\theta=1, 2, 3,\dots$.  As time evolves,
a fraction $\Phi$ of all the nodes are activated. The emergence of
macroscopic properties of interest typically involve $\Phi$ to be at or close to 1.

Gersho and Mitra~\cite{gersho-mitra} studied a similar model for adoption of new communication services using a random regular
graph and obtained (implicit) critical thresholds for widespread adoption.
Chalupa~{\it et al}~\cite{chalupa-1979-bootstrap} were the first to introduce bootstrap 
percolation formally to explain ferromagnetism.
Their analysis is carried out on regular trees (Bethe lattices) and 
a fundamental recursion is derived for computation of the critical threshold
that has since been used extensively.
In the more recent past, results for non-regular (infinite)
trees have also been derived by Balogh~{\it et al}~\cite{balogh-2006-bootstrap}.
Aizenman and Lebowitz~\cite{aizenman-1988-metastability} studied metastability of 
bootstrap percolation on the $d$-dimensional Euclidean lattice $\Z^d$
which has now been thoroughly investigated in two
and three dimensions, see~\cite{holroyd-2003-sharp,cerf-cirillo}.
The existence of a sharp metastability threshold for
bootstrap percolation in two-dimensional lattices was proved
by Holroyd~\cite{holroyd-2003-sharp} and recently generalized to $d$-dimensional 
lattices by Balogh~{\it et al}~\cite{balogh-2011-sharp}. Even more recently,
bootstrap percolation has been studied on random graphs $G(n,p)$ by Luczak~{\it et al}~\cite{luczak-bootstrap}.
In \cite{watts-2002} Watts proposed a model of formation of opinions
in social networks in which the percolation threshold is a certain
fraction of the size of each neighborhood rather than a fixed value,
a departure from the standard model that is  
used by Amini in~\cite{amini-2010-bootstrap} for random graphs with a given degree sequence.

Many diffusion processes of interest have a
{\it physical contact} element.
A link in an {\it ad hoc} wireless network, a sensor network, or
an epidemiological graph
connotes physical proximity within a certain locality. Study of diffusion of
virus spread in {\it ad hoc} wireless, sensor or epidemiological graphs requires this notion of
neighborhood for accurate estimation of likelihood of full percolation.  This is in contrast
to models with long-range reach where physical proximity plays little, if any,
role.
The natural random model for such phenomena is the random geometric graph.
In this work, we focus on bootstrap percolation on random geometric graphs,
a topic that has not been investigated, to the best of our knowledge,
and obtain tight bounds on their critical thresholds for full percolation.

\section{Random Geometric Graph Model}
\label{sec:rgg.model}
One of the transitions from the random graph model $G(n,p)$ of Erd\H{o}s and R\'enyi~\cite{erdos-1959-random, erdos-1960-evolution} and Gilbert~\cite{gilbert-1959-random} to models that may describe processes constrained by geometric distances among the nodes is the model of random geometric graphs (RGGs) by Gilbert \cite{gilbert-1961-random}. 
The RGG model has been used in many disciplines:
for modeling of wireless sensor networks~\cite{pottie-2002-wireless},
cluster analysis, statistical physics, hypothesis testing,
spread of computer viruses in wired networks, 
processes involving physical contact among individuals, 
as well as other related disciplines,
see~\cite{penrose:book} for more details.
For example, a wireless sensor network typically contains a large number of randomly deployed nodes with 
links determined by geometric proximity enabled by (a small) radio range among the nodes that is 
sufficient to enable successful signal transmission across the network.  
A further application of RGG is in representing $d$-attribute data where numerical attributes 
are used as coordinates in $\R^d$ and two nodes are considered connected if they are within a 
threshold (Euclidean) distance $r$ of each other.  
The metric distance imposed on such a RGG captures the similarity between data elements.

Consider an RGG in two dimensions that is constructed by drawing $n$
nodes uniformly
at random within $[0,1]^2$ and connecting every pair of nodes at
Euclidean distance at most $r$.  Let us denote this process by
$\textrm{RGG}(n,r)$.
%
A summary of basic structural properties of $\textrm{RGG}(n,r)$ is as follows. 
\begin{enumerate}
\item
$\textrm{RGG}(n,r)$ is a `homogeneous' geometrical model where the distribution
of the number of nodes within a distance $r$ from a given node follows the
same binomial distribution $\bin(n-1,r^2\pi)$ (with appropriate
correction when the center is within a distance $r$ of the boundary). 
The average degree $D=\E(\deg)$ of a node is $n r^2 \pi$ in the limit.
\item
There is a \textit{critical value} $\lambda_c$ such that
for $r > \sqrt{\lambda_c/n}$ there exists a \textit{giant component},
i.e., the largest connected component of order $\Th(n)$ nodes contained
in $\textrm{RGG}(n,r)$ whp\footnote{Whp or ``with high probability'',
means with probability one as $n$, the number of nodes, tends
to infinity.},~\cite{penrose:book}. We denote the critical threshold
for existence of a giant component by $r_c:=\sqrt{\lm_c/n}$.
\item
In this regime, the second largest component is of order $\Oh(\ln^2 n)$. 
\item
The exact theoretical value of the constant $\lambda_c$ is not known. 
It is experimentally established that
$\lambda_c \approx 1.44$ for the dimension
$d=2$~\cite{RintoulTorquato1997}, while theoretical bounds
$\lambda_c \in [0.696,3.372]$ are given
in~\cite{meester-2003-continuum}. There has been a recent improvement
of the lower bound
$\lambda_c > 4/(3\sqrt{3}) \approx 0.7698$~\cite{kong-2007-analytical}.
\item
$\textrm{RGG}(n,r)$ is connected whp for $r > \sqrt{\ln n /\pi n}$,~\cite{penrose97longest,gupta-1998-critical}. We denote the critical threshold for connectedness by $r_t:=\sqrt{\ln n /\pi n}$.
\item
Every monotone property in a $\textrm{RGG}(n,r)$ (e.g., existence of a giant component and connectedness) exhibits a sharp threshold~\cite{goel-2004-sharp}.
\end{enumerate}

In order to simplify our analysis on RGGs, we now introduce $G_{n,r}$ which is asymptotically isomorphic to $\textrm{RGG}(n,rn^{-1/2})$. Let $\mathcal{X}$ be a Poisson point process of intensity
$1$ on $\mathbb{R}^2$. Consider points of $\mathcal{X}$ contained in
$[0,\sqrt{n}]^2$ representing the nodes of a graph denoted $G_{n,r}$.
Two nodes of $G_{n,r}$ are connected if their Euclidean distance is at
most $r$. Our analysis from here on will be based upon the fact that an instance
of $G_{n,r}$ is isomorphic to an instance of $\textrm{RGG}(n,rn^{-1/2})$
whp~\cite{penrose:book}. 

We parameterize $r = \sqrt{\pi^{-1} a \ln n}$ by introducing a new parameter $a$ which measures how denser $G_{n,r}$ is compared to an instance $G_{n,r_t}$ at the threshold for connectedness $r_t$.
The condition $a>1$ enables us to deal with an asymptotically connected $G_{n,r}$~\cite{gupta-1998-critical,penrose:book}. Notice that for sufficiently large $n$ the expected degree is concentrated around its mean $a \ln n$, which can be easily derived from the Chernoff and union bounds. 

For $n=1000$, the critical thresholds for the existence of a giant component and connectedness in $G_{n,r}$ satisfy $r_c \approx 0.0316$ and $r_t \approx 0.0469$, respectively. 
In Figure~\ref{fig:rggs_1_2} and Figure~\ref{fig:rggs_3_4}, we present $G_{n,r}$ for four different regimes when $r$ takes values: $0.020, 0.035, 0.045, 0.050$, respectively. The values $0.020$ and $0.035$ correspond to `ultra'-sparse regime and emergence of a giant component, Figure~\ref{fig:rggs_1_2}. The values $0.045$ and $0.050$ correspond to `almost'-connected and connected regimes, Figure~\ref{fig:rggs_3_4}.

\begin{figure}[!htb]
\centering
\includegraphics[width=2.45in]{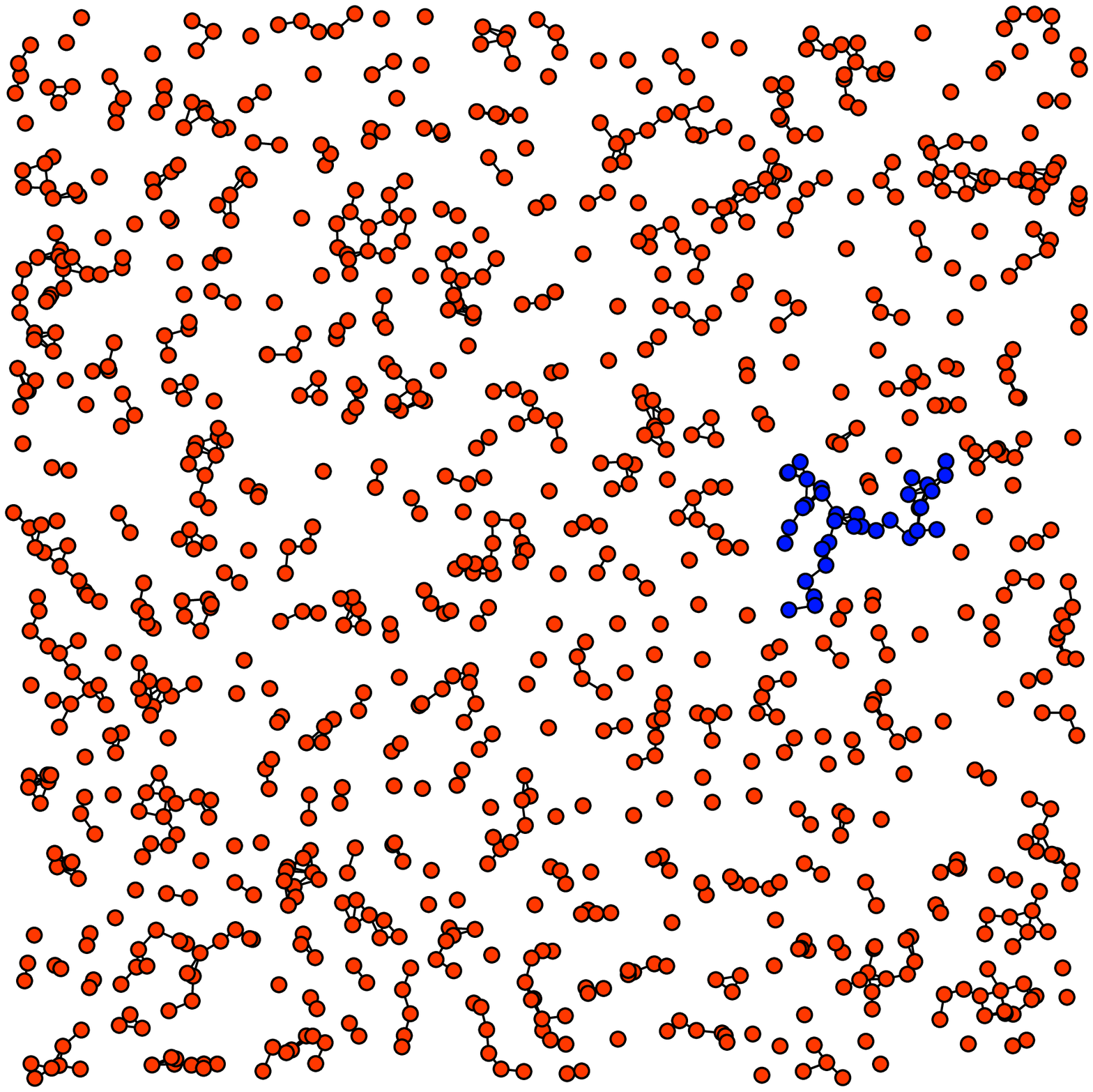}
\includegraphics[width=2.45in]{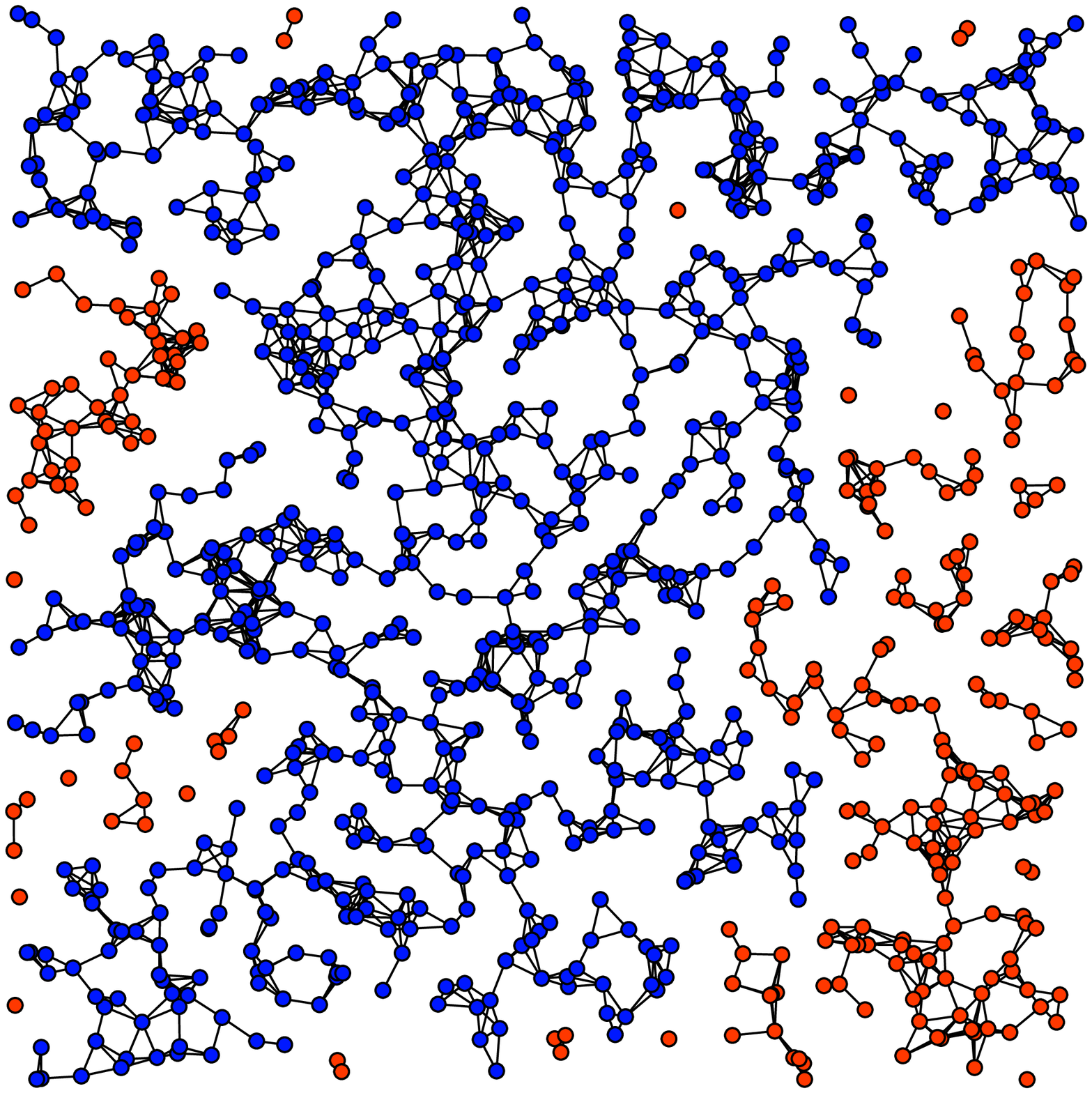}
\caption{`Ultra'-sparse regime and the emergence of the giant component.}
\label{fig:rggs_1_2}
\end{figure}

\begin{figure}[!htb]
\centering
\includegraphics[width=2.45in]{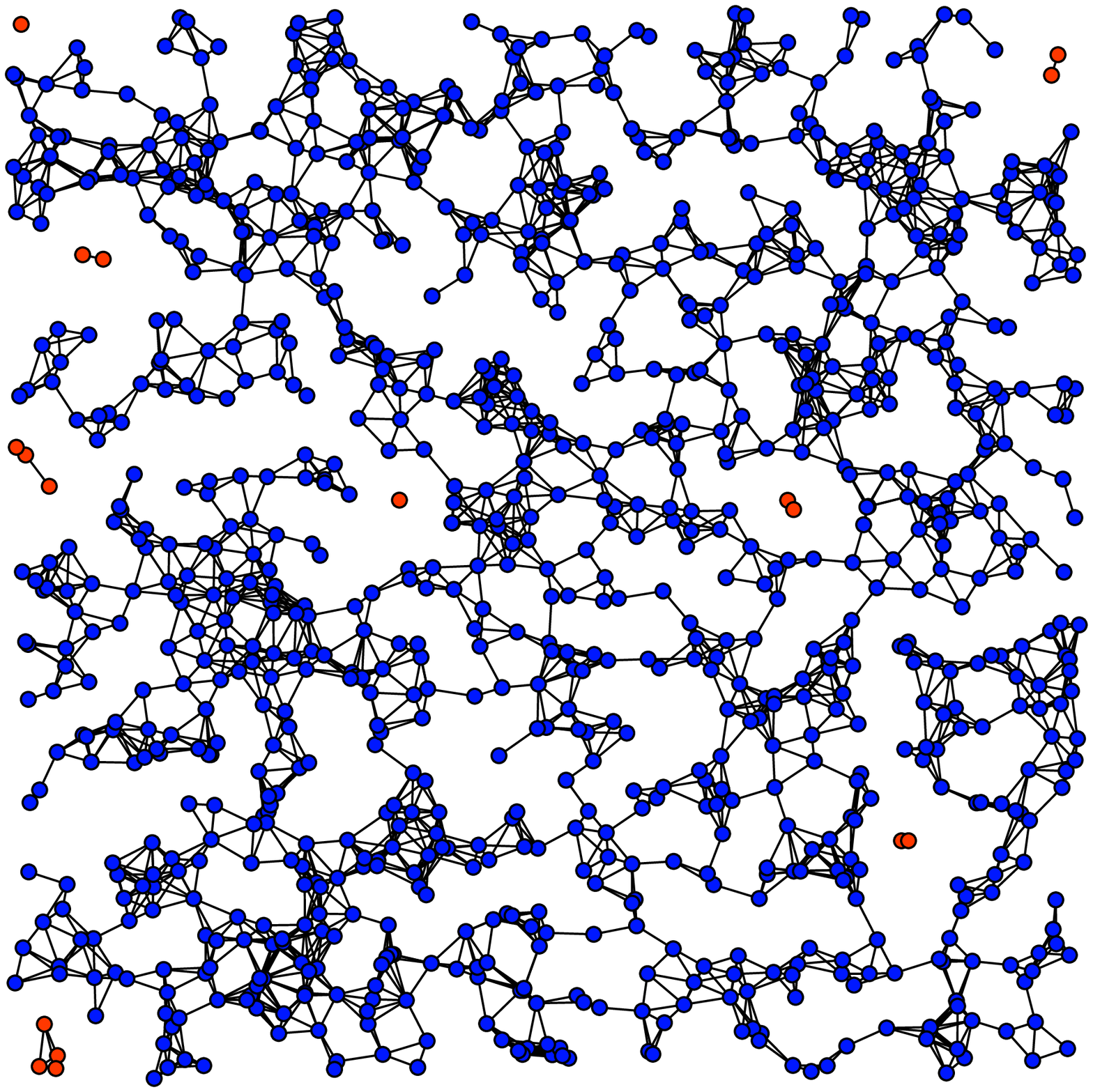}
\includegraphics[width=2.45in]{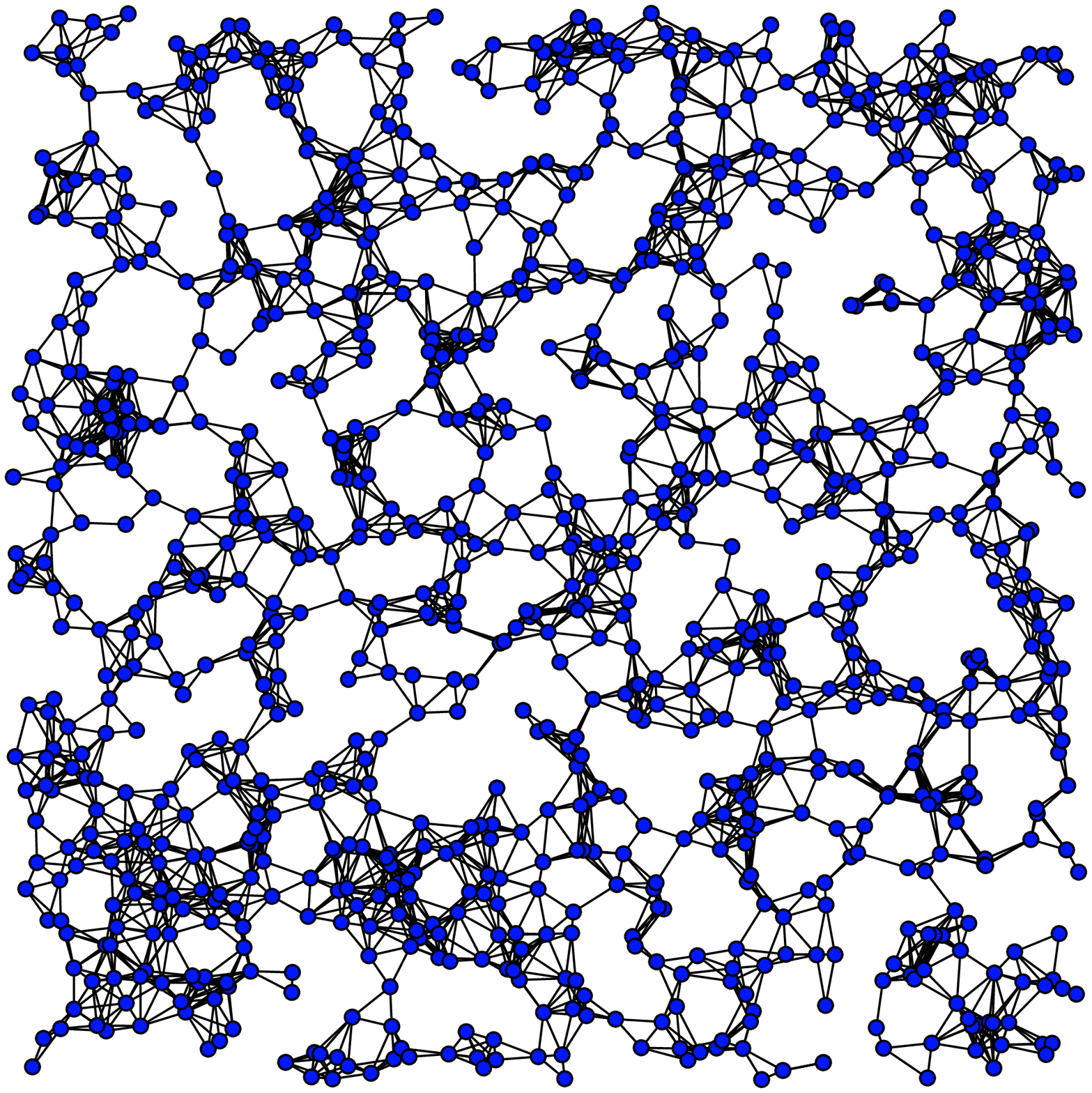}
\caption{`Almost'-connected and connected regimes.}
\label{fig:rggs_3_4}
\end{figure}

\section{Bootstrap Percolation}
\label{sec:bootstrap}

Bootstrap percolation (BP) is a cellular automaton defined on an underlying graph $G=(V,E)$ with state space $\{0,1\}^{V}$ whose initial configuration is chosen by a Bernoulli product measure. In other words, every node is in one of two different states $0$ or $1$ ({\it inactive} or {\it active} respectively), and a node becomes active with probability $p$ independently of other nodes within the initial configuration. 

After drawing an initial configuration at time $t=0$, a discrete time deterministic process updates the configuration according to a local rule:
an inactive node becomes active at time $t+1$ if the number of its active neighbors at $t$ (not necessarily the nearest ones) is greater than some defined {\it threshold} $\theta$. Once an inactive node becomes active it remains active forever. A configuration that does not change at the next time step is a {\it stable} configuration.
A configuration is {\it fully active} if all its nodes of are active.

An interesting phenomenon to study is metastability near a first-order phase transition. 
Do there exist $0< p_c' \leq p_c''<1$ such that: 
\beq
\nonumber
\left(\forall p<p_c'\right) \lim_{t \to \infty} \pr_p \left(V \textrm{ becomes fully active}\right) = 0\,,
\eeq
and 
\beq
\nonumber
\left(\forall p>p_c''\right) \lim_{t \to \infty} \pr_p \left(V \textrm{ becomes fully active}\right) =1 \,?
\eeq
Further, is it necessary for $p'_c$ to be asymptotically equal to $p''_c$?

A study of BP on a regular infinite tree first appeared in~\cite{chalupa-1979-bootstrap}. Subsequently, the relations between the branching number of an infinite (non-regular) tree, threshold value, and $p$ necessary to fully percolate the tree were studied in~\cite{balogh-2006-bootstrap}.

An example of BP is a $d$-dimensional lattice $\Z^d$ equipped with Bernoulli product measure with $\th=d$~\cite{aizenman-1988-metastability}. For $\Z^d$ and $V = [0,L-1]^d$ the existence of a unique threshold $p_c$ was shown in~\cite{aizenman-1988-metastability}. Concretely for $\Z^2$ and $V = [0,L-1]^2$ the exact threshold value is $p_c = \pi^2/(18 \ln L)$~\cite{holroyd-2003-sharp}. Furthermore the sharp threshold for bootstrap percolation in $\Z^d$ in all dimensions was provided in~\cite{balogh-2011-sharp}.

Additionally to BP on trees and lattices, there has been recent work of BP on random regular graphs~\cite{balogh-2007-bootstrap}, Erd\H os-R\'enyi random graphs~\cite{luczak-bootstrap}, as well as random graphs with a given degree sequence where the threshold depends upon node degree~\cite{amini-2010-bootstrap}.

\subsection{Bootstrap Percolation on Connected RGGs}
\label{sec:bootstrap.conn.rggs}

The structure of $G_{n,r}$ is conducted by random positions of its nodes and radius $r=r(n)$; so it is more `irregular' than the structure of a tree or a lattice.
In this work we are interested in BP on $G_{n,r}$ which for brevity we denote by $BP(G_{n,r},p,\theta)$. In this process a node becomes active with probability
$p$ independently of other nodes in the initial configuration and
an inactive node becomes active at the following time step if at least
$\theta=\gm D$ of its neighbors are active, where $\gm=\gm(n)$ and
$D(n)= \E(\deg) = r^2 \pi = a \ln n$ is the expected node degree.

For the critical thresholds $p_c'$ and $p_c''$ in $BP(G_{n,r},p,\theta)$, we derive bounds $p'\leq p_c'$ and $p'' \geq p_c''$ such that a connected $G_{n,r}$ does not become fully active for $p < p'$ whp, and conversely, becomes fully active for $p > p''$  whp. These bounds are schematically presented in Figure~\ref{fig:bounds}.

\begin{figure}[!htb]
\centering
\includegraphics[width=0.55\columnwidth]{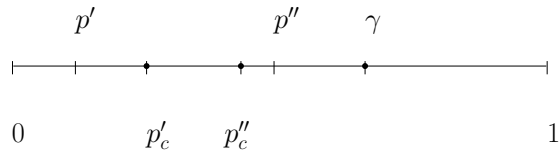}
\caption{Bounds ($p',p''$) on the critical thresholds ($p'_c,p''_c$).}
\label{fig:bounds}
\end{figure}

The main ideas of the proofs are as follows. We obtain the distribution of the number of active neighbors for each node at the initial configuration. 
For $p<p'$ we use the Poisson tail bound and the union bound (see~(\ref{eq:po1}) in Appendix) to show that an initial configuration is stable whp. 
For $p>p''$ we use the Bahadur-Rao theorem (see Claim~\ref{cl:br.poiss} in Appendix) to lower bound the number of active neighbors for each node. Then we develop a geometric argument to show that a stable, fully active, configuration is reached within $\Oh(\sqrt{n}/r)$ steps whp. This geometric argument leverages the following simple observation about BP in $\Z^2$ with $\th=1$.

\begin{lemma}
\label{lemma:z2th1}
Consider BP in $\Z^2$ with the threshold $\theta=1$ and the initial probability $p>0$. For any $N$ and $p=\omega(1/\sqrt{N})$, a square $[0,N]^2$ becomes fully active within $\Oh(N)$ steps whp.   
\end{lemma}

We first introduce the following functions upon which our analysis will heavily depend. For the function $H(x):= x \ln x - x +1$ on $[0,+\infty)$ (see Figure~\ref{fig:HJ} left), define $H_{L}^{-1}:[0,1] \to [0,1]$ to be the inverse of $H(x)$ on $[0,1]$, and $H_{R}^{-1}:[0,+\infty) \to [1,+\infty)$ to be the inverse of $H(x)$ on $[1,+\infty)$. Analogously for the function $J(x):=x^{-1}H(x)=\ln x - 1 + x^{-1}$ on $(0,+\infty)$ (see Figure~\ref{fig:HJ} right), define $J_{L}^{-1}:[0,+\infty] \to [0,1]$ to be the inverse of $J(x)$ on $[0,1]$, and $J_{R}^{-1}:[0,+\infty) \to [1,+\infty)$ to be the inverse of $J(x)$ on $[1,+\infty)$.

\begin{figure}
\centering 
\begin{minipage}[htb!]{0.485\linewidth}
\includegraphics[width=7cm]{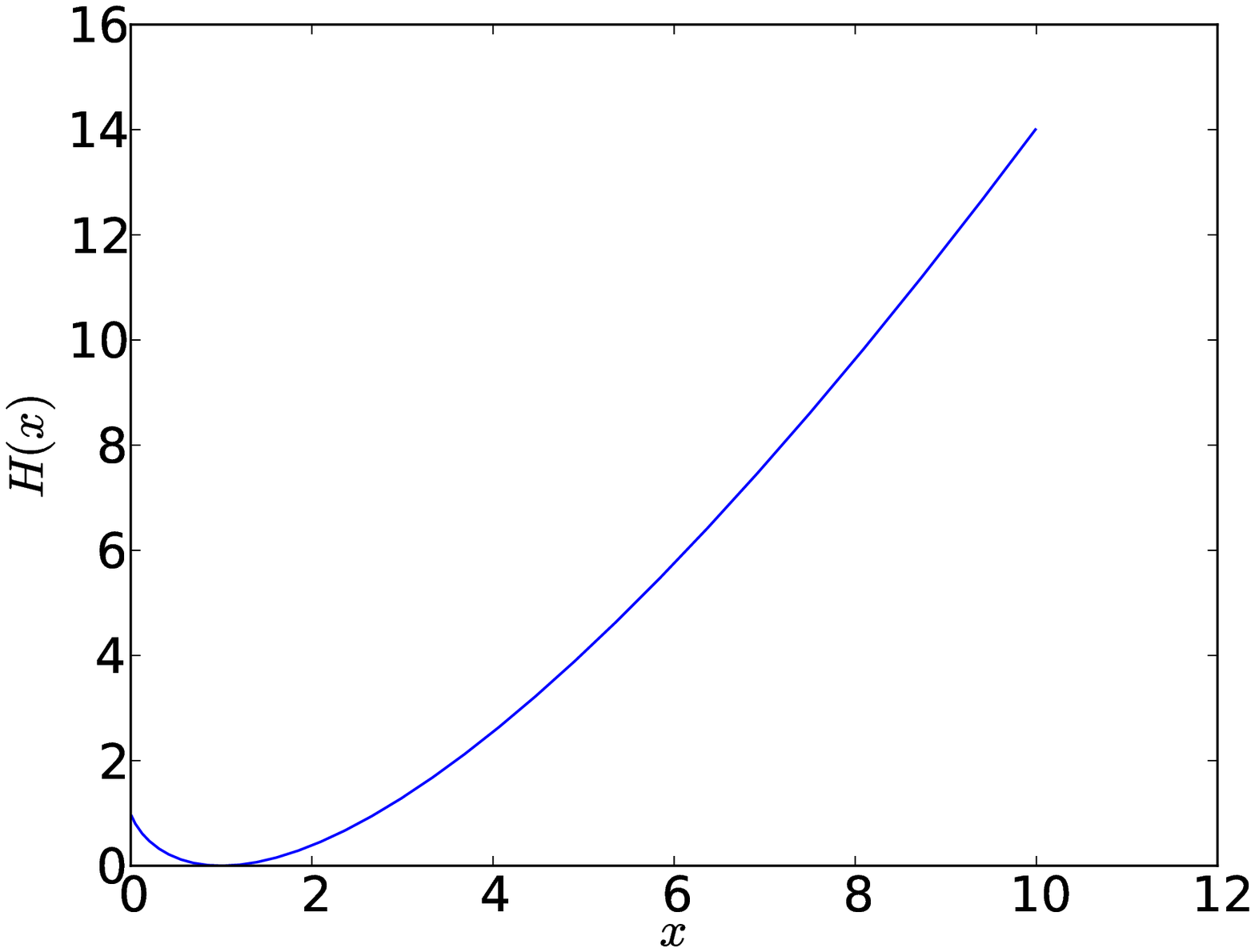}
\end{minipage}
\begin{minipage}[htb!]{0.485\linewidth}
\includegraphics[width=7cm]{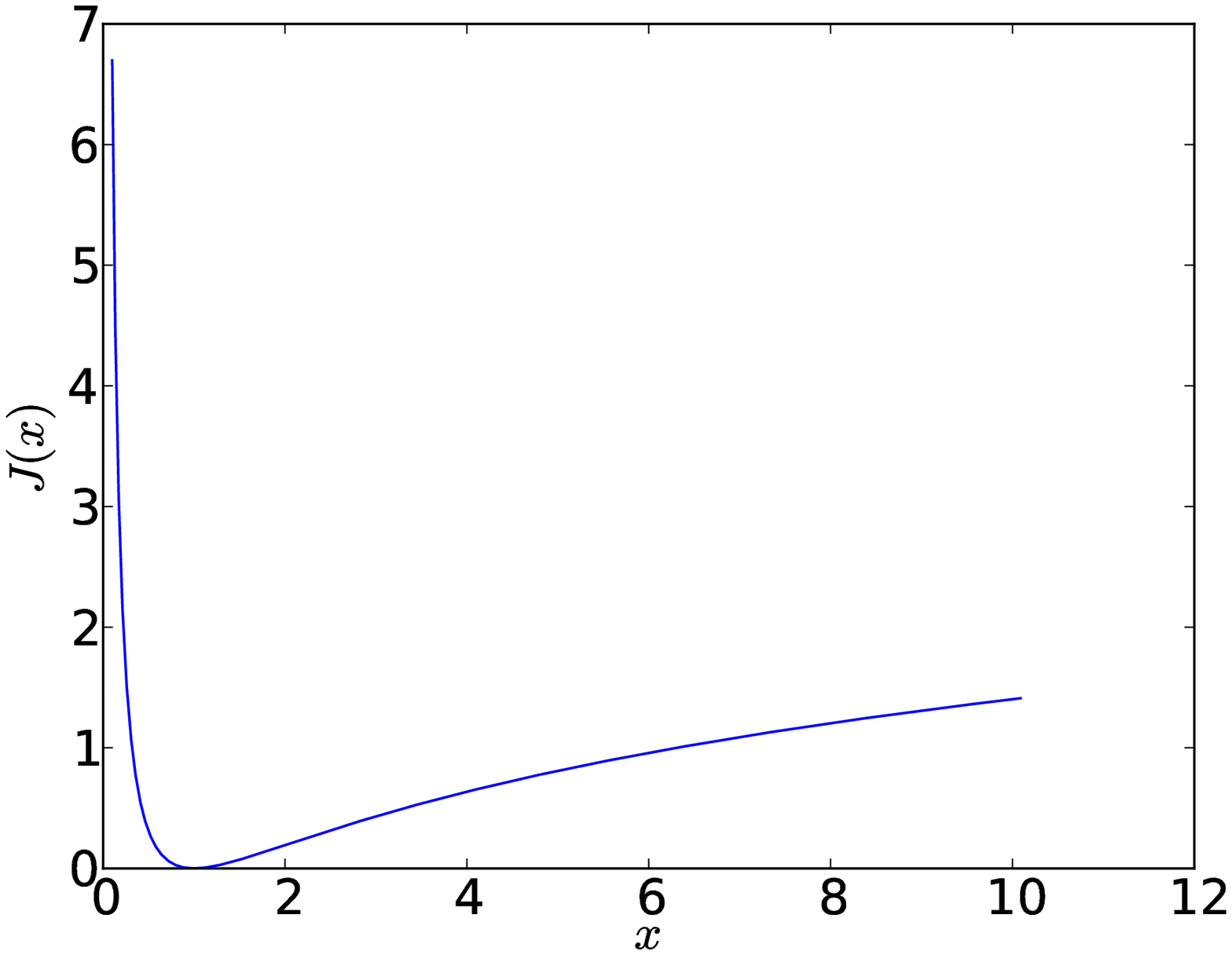}
\end{minipage}
\caption{Functions $H(x)$ and $J(x)$.}
\label{fig:HJ}
\end{figure}

We now provide bounds on the critical thresholds on $p$ in Theorem~\ref{thm:incomplete.percolation} and Theorem~\ref{thm:complete.percolation}.

\begin{theorem}
\label{thm:incomplete.percolation}
Consider bootstrap percolation $BP(G_{n,r},p,\theta)$ where $r=\sqrt{\pi^{-1} a \ln n}$ and $\theta=\gm a \ln n$. For $a>1$, $\gm \in (0,1)$ and 
when \begin{equation*}
p < p':=\gm/J_R^{-1}(1/a\gm) \,, 
\end{equation*}
$G_{n,r}$ does not become fully active whp.
\end{theorem}

\begin{proof}
We show that for the conditions of the assertion, an initial configuration is stable. The number of active nodes in the initial configuration follows Poisson distribution $\po(pn)$. The degree distribution of a node is $\po(r^2 \pi)-1$, and the expected degree $D = r^2 \pi = a \ln n$. By the thinning theorem~\cite{penrose:book} the number of active neighbors in the initial configuration follows $\po(p D) - 1$. Consider the activation rule in $BP(G_{n,r},p,\theta)$. The probability that a node becomes active at the next time step given that it is inactive initially is $\pr \left( \po(p D) - 1 \geq \gm D \right)$. 

For $p>\gm$, given that $p D \to \infty$, the tail bound on a Poisson random variable~(\ref{eq:po1}) implies for any node $\pr \left( \po(p D) - 1 \geq \gm D \right) = 1- o(1/n)$. Hence an initial configuration becomes fully active at the next time step with probability $1-o(1)$. Therefore we consider the case $p \leq \gm$ and seek a maximal $p' \leq \gm$ (see Figure~\ref{fig:bounds}) such that $BP(G_{n,r}, p', \gamma)$ does not become fully active whp. It follows
\beq
\label{eq:btstp.conditions}
\pr \left( \po(p D) - 1 \geq \gm D \right) \leq \pr \left( \po(p D) \geq \gm D \right) \leq \exp(-p D H(\gm/p)) \,.    
\eeq
%
The same inequality~(\ref{eq:po1}) yields that the number of nodes $\po(n)$ within the square $[0,\sqrt{n}]^2$ is concentrated around its mean $n$ whp. Hence the union bound over all nodes provides
\beq
\label{eq:stable.conf}
\pr_p \left( \textrm{the initial configuration is stable} \right) \geq 1 - \exp\left((1+\o(1)) \ln n -p D H(\gm/p) \right) \,.    
\eeq
Given $D=a\ln n$, the condition $p a H(\gm/p)>1$ suffices that the initial configuration is stable whp. The function $J(x)=x^{-1}H(x)$ is monotonically decreasing on $(0,1)$, monotonically increasing on $(1,+\infty)$, with the minimum $0$ attained at $x=1$. Hence for any positive $\gm < +\infty$ there are two solutions of $J(x) = 1/a \gm$, denoted $x_1 < 1 < x_2$. This yields $p>\gm/x_1>\gm$ or $p<\gm/x_2<\gm$. The acceptable solution is $p<\gm/x_2$, since we consider the case $p<\gm$. For $J(\gm/p) > 1/a\gm$ from (\ref{eq:stable.conf}) it follows the probability that the initial configuration is stable tends to one as $n$ tends to infinity. Finally, a bound on $p$ is given by 
\beq
\nonumber
p < p':= \gm/J_R^{-1}(1/a\gm)\,, 
\eeq
which concludes the proof. 
\end{proof}

The following result clarifies the feasible region for $a$ and $\gamma$ in Theorem~\ref{thm:complete.percolation}.
\begin{lemma}
\label{lm:ag}
The condition $a \geq 5 \pi /H(5 \pi \gamma)$ is equivalent to: 
\begin{equation*}
\gamma \in \left[0,\frac{1}{5\pi} H_{R}^{-1} \left( 5\pi/a \right)\right], \quad \textrm{ for }a<5 \pi \,, \\
\end{equation*}
and 
\begin{equation*}
\gamma \in \left[\frac{1}{5\pi} H_{L}^{-1} \left( 5\pi/a \right),\frac{1}{5\pi} H_{R}^{-1} \left( 5\pi/a \right)\right], \quad \textrm{ for } a \geq 5 \pi \,.
\end{equation*}
\end{lemma}

\begin{proof}
By inspection of the function $H(x)$. 
\end{proof}

\begin{theorem}
\label{thm:complete.percolation}
Consider bootstrap percolation $BP(G_{n,r},p,\theta)$ where $r=\sqrt{\pi^{-1} a \ln n}$, $\theta=\gm a \ln n$, and $a>1$.
When $a \geq 5 \pi /H(5 \pi \gamma)$ and $\gamma \in (0,1/5\pi)$ for 
\begin{equation*}
p>p'':= \min \left\{ \gamma, \frac{5\pi \gamma} { J_{R}^{-1}\left( 1/ a \gamma \right)} \right\} \,,
\end{equation*}
$G_{n,r}$ becomes fully active within $\Oh(\sqrt{n}/r)$ steps whp. 
\end{theorem}

For $p>\gm$ an initial configuration becomes fully active at the next time step whp (see the proof of Theorem~\ref{thm:incomplete.percolation}), therefore we consider the case $p \leq \gm$.

The proof of Theorem~\ref{thm:complete.percolation} consists of two parts. We tile the square $[0,\sqrt{n}]^2$ into cells $r / \sqrt{5} \times r / \sqrt{5}$ and show that in the initial configuration: (i) When $a>1$, $\gamma \in (0,1/5\pi)$, and $a \geq 5 \pi /H(5 \pi \gamma)$, every cell contains at least $\gm D$ nodes whp; (ii) When $p>p''$ at least one cell contains $\gm D$ or more active nodes. By Lemma~\ref{lemma:z2th1} it follows that for $a,\gm,p$ in the specified ranges $G_{n,r}$ becomes fully active within $\Oh(\sqrt{n}/r)$ steps whp.

\begin{proof}
Tile the square $[0,\sqrt{n}]^2$ into cells $r / \sqrt{5} \times r / \sqrt{5}$, see Figure~\ref{fig:grid}. 
Define the area of a cell $A:=r^2/5=a\ln n/5\pi$.
Call two cells neighboring if they share one side. Notice every pair of nodes within the same cell or within two neighboring cells are adjacent by the choice of the size of a cell. Define $G'_{n,r}$ on the set of nodes of $G_{n,r}$ as follows. The set of edges of $G'_{n,r}$ consists of the subset of edges of $G_{n,r}$ whose terminal nodes belong to the same cell or two neighboring cells. Then the monotonicity of bootstrap percolation yields
\beq
\label{eq:mono}
\pr_p \left(G'_{n,r} \textrm{ becomes fully active} \right) \leq \pr_p \left( G_{n,r} \textrm{ becomes fully active}\right)\,.
\eeq
Therefore it is sufficient to show that whp $G'_{n,r}$ becomes fully active when $p>p''$~(\ref{eq:2nd:bnd}).

\begin{figure}[!h]
\centering
\includegraphics[width=0.70\columnwidth]{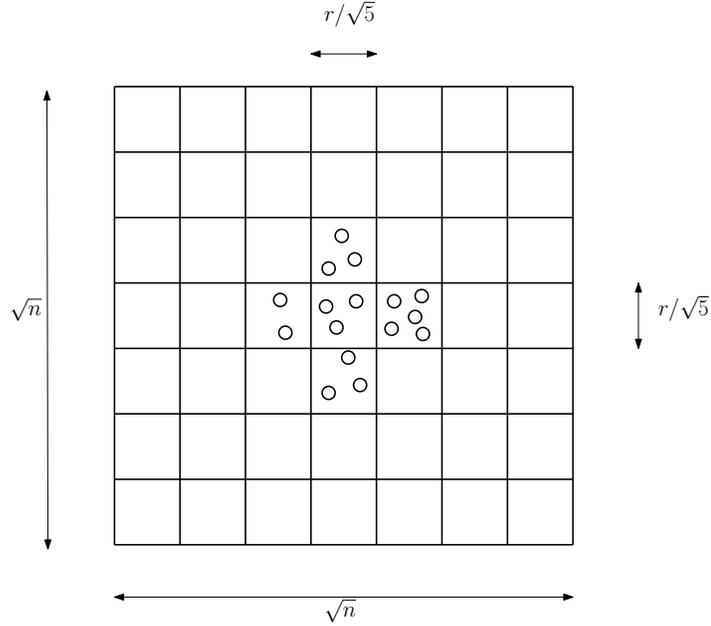}
\caption{Tiling the square $[0,\sqrt{n}]^2$.}
\label{fig:grid}
\end{figure}

Part (i) (To show that every cell contains at least $\gm D$ nodes whp.) We first bound the probability that an arbitrary cell contains at most $\gm D$ nodes. The number of nodes in a cell follows $\po(A)$, i.e., $\po(a\ln n/5 \pi)$. Moreover, the numbers of nodes in cells are independent random variables (given the Poisson point process $\mathcal{X}$). For $\gm \leq 1/5\pi$, from~(\ref{eq:po1}) we obtain  
\beqn
\nonumber
\pr \left( \textrm{a cell contains at most $\gm D$ nodes} \right) &=& \pr \left( \po\left( \frac{a \ln n}{5 \pi} \right) \leq \gm a \ln n \right) \\
\nonumber
&\leq& \exp \left( - \frac{a \ln n}{5 \pi} H (5\pi \gm)\right) \\
\nonumber
&=& n^{-\frac{a}{5\pi}H(5\pi \gm)}\,. 
\eeqn
The total number of cells in $[0,\sqrt{n}]^2$ is $5n/r^2= 5 \pi n/(a \ln n) = \o(n)$. 
The union bound taken over all cells yields
\beq
\label{eq:every.cell.dense.enough}
\pr \left( \textrm{every cell contains at least $\gm D$ nodes} \right) \geq 1 - \o\left(n^{1-\frac{a}{5\pi}H(5\pi \gm)}\right) \,.
\eeq
Finally, for $a \geq 5\pi / H (5 \pi \gm)$, from~(\ref{eq:every.cell.dense.enough}) it follows that every cell contains at least $\gm D$ nodes whp.

Part (ii).
(To show that at least one cell contains $\gm D$ or more active nodes.)
We now derive conditions such that at least one cell contains at least $\theta=\gm D$ active nodes in the initial configuration. In order to guarantee that whp there is at least one cell among $5n/r^2=\Th(n / \ln n)$, which contains at least $\th$ active nodes in the initial configuration, it suffices to find $p$ such that
\beq
\label{eq:omega}
\pr \left( \po(p A) \geq \gm D + 1 \right) = \omega \left( \frac{\ln n}{n} \right) \,,   
\eeq
since 
\beq
\nonumber
\lim_{n \to \infty}1-\left(1-\omega(\ln n /n)\right)^{\Th(n / \ln n)} = 1 \,.
\eeq
Define $\al := 5 \pi \gm/p - 1$, then by rewriting (\ref{eq:omega}) we need $p$ such that
\beq
\label{eq:FBP}
\pr \left( \frac{\po(p A) -p A}{p A} \geq \al + \frac{1}{pA} \right) = \omega \left( \frac{\ln n}{n} \right) \,. 
\eeq
By the Bahadur-Rao tail bound~\cite{bahadur-deviations-1960}, as $n \to \infty$, i.e., $p A = \Th(p\ln n) \to \infty$, the shifted Poisson random variable $\po(p A) -p A$ satisfies
\beq
\nonumber
\pr \left( \frac{\po(p A) -p A}{p A} \geq \al + \frac{1}{pA} \right) \approx \frac{\sqrt{1+\al}}{\al \sqrt{2 \pi}} \frac{1}{\sqrt{p A}} \exp \left(- p A I(\al) \right) \,,
\eeq
where the {\it rate function} is defined by
\beq
\nonumber
\label{eq:rate}
I(\al) = \sup_{s \in \mathbb{R}}\left\{s \al - e^{s} + s + 1\right\} = (1+\al)\ln(1+\al) - \al \,. 
\eeq
(See Appendix for the details.)
Therefore for (\ref{eq:FBP}) to be satisfied we require 
\beq
\label{eq:cond}
\frac{n}{\ln n} \frac{\sqrt{1+\al}}{\al \sqrt{2 \pi}} \frac{1}{\sqrt{p A}} \exp \left(- p A I(\al) \right) = \omega(1) \,.
\eeq
The left hand side of (\ref{eq:cond}) equals 
\beqn
\nonumber
\nonumber
&=&  \exp \left(\left( 1- \frac{ap}{5\pi} I(\al) \right) \ln n - \frac{3}{2}\ln \ln n -\frac{1}{2}\ln\frac{pa}{5\pi} + \ln\frac{\sqrt{1+\al}}{\al \sqrt{2 \pi}} \right) \,,
\eeqn
which is $\omega(1)$ if $1>a p I(\al)/5\pi$. 
Given $\al = 5 \pi \gm/p - 1$, the condition $1>a p I(\al)/5\pi$ is equivalent to $1/a\gm > H(5\pi\gm/p)/(5\pi\gm/p)$, and moreover to
\beq
\label{eq:2nd:bnd}
p > p'':=\frac{5\pi\gm}{J_R^{-1}(1/a\gm)}\,.
\eeq

To complete the proof notice that once any $\gm D$ nodes within a cell become active, all nodes within that cell become active at the next time step as would all nodes within its neighboring cells. This resulting process which jointly activates all nodes within one cell is equivalent to activating a site in $\Z^2$. The resulting BP in $\Z^2$ has the threshold $\th=1$ by construction, see Figure~\ref{fig:grid}. Thus $BP(G'_{n,r},p,\theta)$ becomes fully active when $p>p''$ by Lemma~\ref{lemma:z2th1}. The proof follows from~(\ref{eq:mono}).
\end{proof}

\begin{rem}
For non-trival percolation threshold, that is, $p'' \leq \gamma$, it is necessary
\begin{equation*} 
a \gamma \leq \frac{1}{J_R(5\pi)} \approx 0.55 \,.
\end{equation*} 
\end{rem}

\begin{rem}
When $a>1$, the upper bound on $\gamma$ in Lemma~\ref{lm:ag} is further tightened: 
\begin{equation*}
\gamma \in \left[0,\frac{1}{aJ_R(5\pi)}\right], \quad \textrm{ for }a<5 \pi \,, \\
\end{equation*}
and 
\begin{equation*}
\gamma \in \left[\frac{1}{5\pi} H_{L}^{-1} \left( 5\pi/a \right),\frac{1}{aJ_R(5\pi)}\right], \quad \textrm{ for } a \geq 5 \pi \,.
\end{equation*}
\end{rem}

\subsection{Analysis of Bounds on Critical Thresholds}

The critical threshold $p' = \gm/J_R^{-1}(1/a \gm)$ can be rewritten as 
\beq
\label{eq:leveled_logpc}
\ln p' = -\ln a - \ln \left((1/a\gm)J_R^{-1}(1/a\gm)\right)\,. 
\eeq

The function $-\ln \left(x J_R^{-1}(x) \right)$ is monotonically decreasing in $x$, hence $p'$ is monotonically increasing in $a$ and monotonically decreasing in $\gm$.
As an example we numerically compute and tabulate $p'$ for $\gm=1/20$ and different values of $a$ in Table~1.
In Figure~\ref{fig:pc_vs_a_vs_gm}, $p'$ is plotted as a function of $a$ for different values of 
\beq
\nonumber
\gm \in \{1/70, 1/60, 1/50, 1/40, 1/30, 1/20\} \,.
\eeq

\begin{table}[!ht]
\centering
\begin{tabular}{|c|c|c|}
\hline
$a$ & $p'$ & $p''$\\
\hline
\hline
3 & 0.0000234198 & 0.0003678767\\ 
\hline
4 & 0.0001242460 & 0.0019516511\\ 
\hline
5 & 0.0003391906 & 0.0053279940\\ 
\hline
6 & 0.0006649716 & 0.0104453500\\ 
\hline
7 & 0.0010794693 & 0.0169562642\\ 
\hline
8 & 0.0015576467 & 0.0244674579\\ 
\hline
9 & 0.0020779022 & 0.0326396121\\ 
\hline
10 & 0.0026234549 & 0.0412091329\\ 
\hline
25 & 0.0101188498 & 0.1589465210\\ 
\hline
50 & 0.0174952121 & 0.0174952120\\ 
\hline
100 & 0.0246619916 & 0.3873896589\\ 
\hline
\end{tabular}
\label{table:pc}
\caption{Bounds $p',p''$ on the critical thresholds for different values of $a$ when $\gm=1/20$.}
\end{table}

\begin{figure}[!ht]  
\centering 
\includegraphics[width=0.625\columnwidth]{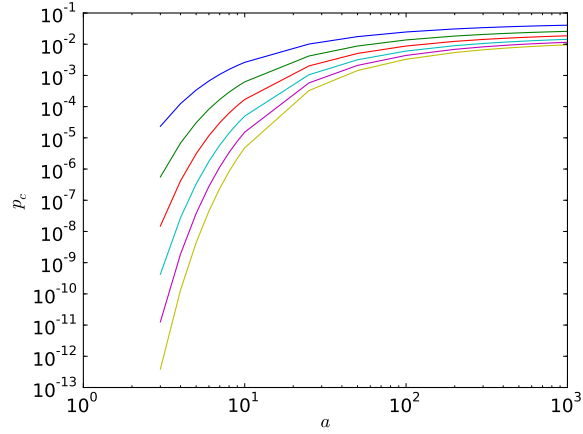}
\caption{The bound $p'$ for $\gm \in \{1/70, 1/60, 1/50, 1/40, 1/30, 1/20\}$ as a function of $a$.}
\label{fig:pc_vs_a_vs_gm}
\end{figure}

The experiments are performed on $G_{n,r}$ with $n=15000$ and $n=25000$ nodes, and $r=\sqrt{a \ln n /\pi}$ for the cases: (i) $a=30$ and $\gamma=1/100$, and (ii) $a=35$ and $\gamma=1/75$. On these instances of graphs, for each chosen value of $p$ in $(0,1)$ we simulate BP $100$ times. More precisely, within each experiment we generate a random initial configuration with the probability $p$ and perform BP with the threshold $\th =\gm D$ where the expected degree $D$ is calculated for a given input $G_{n,r}$. 

Numerical results are presented with the initial probability $p$ on the horizontal axis, and the percentage of fully active stable configurations on the vertical axis. 
Four cases when $(a=30,\gm=1/100)$, $(a=35,\gm=1/75)$, for $n=15000,25000$, are presented in Figures~\ref{fig:experiments-30-100-15000},~\ref{fig:experiments-30-100-25000},~\ref{fig:experiments-35-75-15000} and~\ref{fig:experiments-35-75-25000}, respectively. These charts match the bounds derived theoretically for $p'$ and $p''$. Further, they appear to support the case that $p'_c\neq p''_c$ even though we do not currently have a proof one way or the other.

\begin{figure}[!htb]
\centering
\includegraphics[width=3.5in]{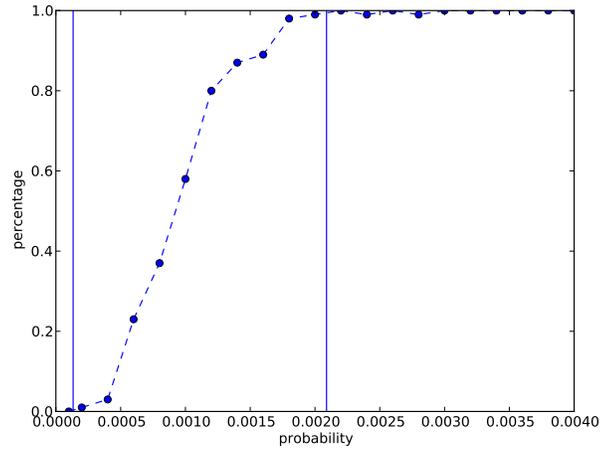}
\caption{Percentage of fully percolated configurations in $100$ simulations of $BP(G_{n,r},p,\th)$ when $a=30, \gamma=1/100$, $n=15000$, $r=\sqrt{30\ln n/\pi n} \approx 0.07824$, $D=30 \ln n \approx 288.47$ and $\th= \lceil 100^{-1}\E(\deg) \rceil = \lceil 2.88 \rceil = 3$. The bounds are $p'=0.000133$ and $p''=0.002089$.}
\label{fig:experiments-30-100-15000}
\end{figure}

\begin{figure}[!htb]
\centering
\includegraphics[width=3.5in]{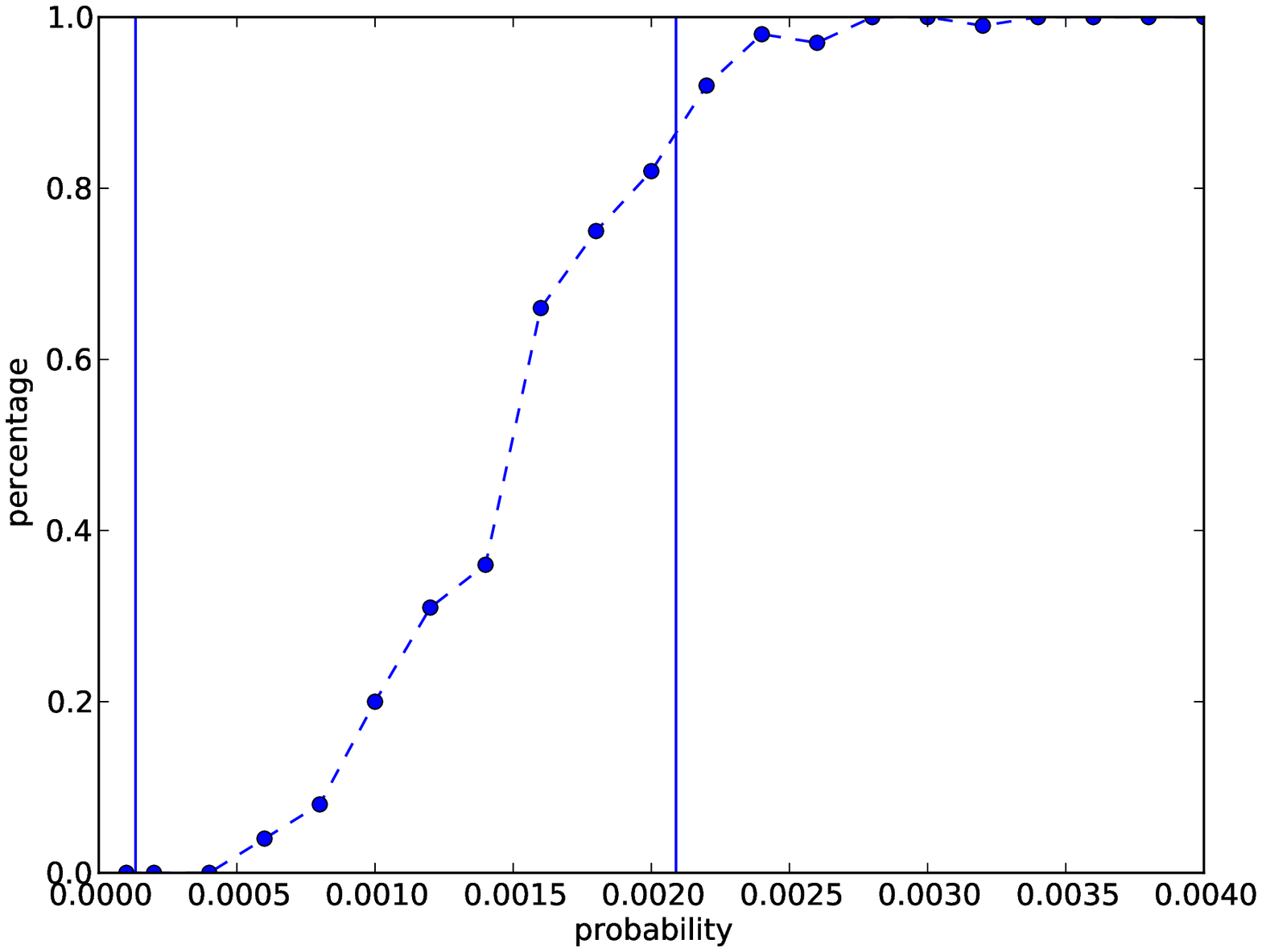}
\caption{Percentage of fully percolated configurations in $100$ simulations of $BP(G_{n,r},p,\th)$ when $a=30, \gamma=1/100$, $n=25000$, $r=\sqrt{30\ln n/\pi n} \approx 0.06219$, $D=30 \ln n \approx 303.80$ and $\th= \lceil 100^{-1}\E(\deg) \rceil = \lceil 3.04 \rceil = 4$. The bounds are $p'=0.000133$ and $p''=0.002089$.}
\label{fig:experiments-30-100-25000}
\end{figure}

\begin{figure}[!htb]
\centering
\includegraphics[width=3.5in]{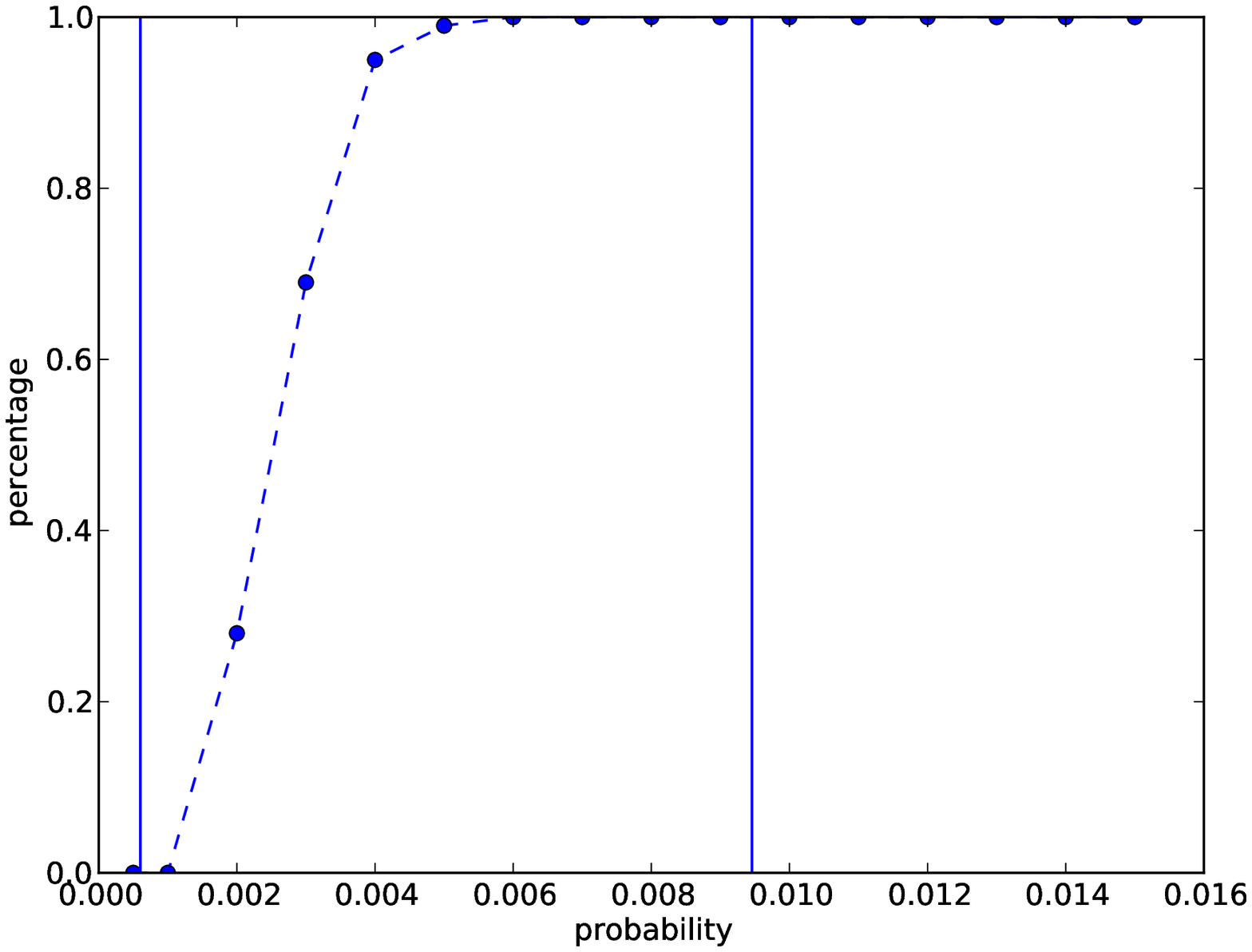}
\caption{Percentage of fully percolated configurations in $100$ simulations of $BP(G_{n,r},p,\th)$ when $a=35, \gamma=1/75$, $n=15000$, $r=\sqrt{35\ln n/\pi n} \approx 0.08451$, $D=35 \ln n \approx 336.55$ and $\th= \lceil 75^{-1}\E(\deg) \rceil = \lceil 4.49 \rceil = 5$. The bounds are $p'=0.000602$ and $p''=0.009457$.}
\label{fig:experiments-35-75-15000}
\end{figure}

\begin{figure}[!htb]
\centering
\includegraphics[width=3.5in]{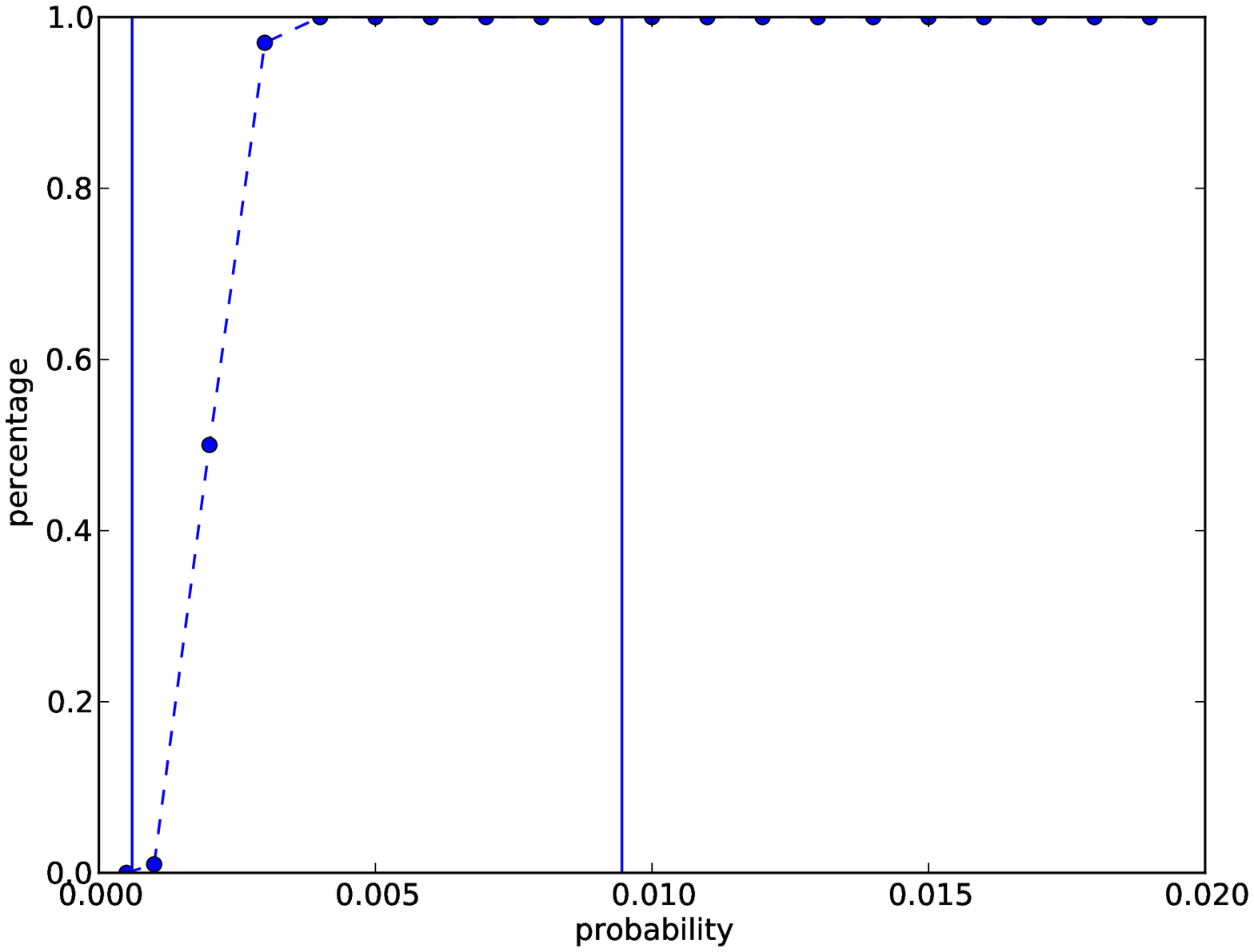}
\caption{Percentage of fully percolated configurations in $100$ simulations of $BP(G_{n,r},p,\th)$ when $a=35, \gamma=1/75$, $n=25000$, $r=\sqrt{35\ln n/\pi n} \approx 0.06718$, $D=35 \ln n \approx 354.43$ and $\th= \lceil 75^{-1}\E(\deg) \rceil = \lceil 4.73 \rceil = 5$. The bounds are $p'=0.000602$ and $p''=0.009457$.}
\label{fig:experiments-35-75-25000}
\end{figure}

\section*{Acknowledgements}
This work was funded by NIST Grant No. 60NANB10D128.

\newpage
\bibliographystyle{acm}
\bibliography{boot}

\section*{Appendix}

\begin{lemma}(Concentration on a Poisson random variable, see~\cite{penrose:book}).
\label{lemma:poisson}
A Poisson random variable $\po(\lm)$ (with $\lm>0$) satisfies: 
\beqn
\label{eq:po1}
\pr\left( \po(\lm) \geq k \right) &\leq& \exp\left( -\lm H(k/\lm)\right), \textrm{ for $k \geq \lm$} \,, \\
\label{eq:po2}
\pr\left( \po(\lm) \leq k \right) &\leq& \exp\left( -\lm H(k/\lm)\right), \textrm{ for $k \leq \lm$} \,, 
\eeqn
where $H(x)=x\ln x -x + 1$ for $x>0$.
\end{lemma}

\noindent
\begin{theorem}(Bahadur-Rao, see~\cite{bucklew-1990-large})
\label{thm:br}
Let $X_1,X_2,\dots$ be an i.i.d. sequence of random variables such that $\E(X_1)=0$ and $M(s) := \E(e^{sX_i}) < \infty$ for all $s \in \R$. 
If $X_1$ is of lattice type and $\pr(X_1=\al)>0$, then
\beq
\nonumber
\lim_{N \to \infty} \pr \left( \frac{1}{N}\sum_{i=1}^N X_i \geq \al \right) \exp \left(N I(\al) \right) \sqrt{N}= \frac{1}{\sigma \sqrt{2 \pi} \left(1-\exp(-s_\al)\right)}\,,
\eeq
where $I(\al) := \sup_{s \in \R} \left(s\al - \ln M(s)\right)$ attained at $s=s_\al$, and $\sigma^2 = M''(s_\al)/M(s_\al)-\al^2$.
\end{theorem}

\begin{claim}
\label{cl:br.poiss}
A Poisson random variable $\po(N)$ for $N \to \infty$ satisfies 
\beq
\nonumber
\lim_{N \to \infty} \pr \left( \frac{\po(N)-N}{N} \geq \al \right) \exp \left(N I(\al) \right) \sqrt{N}= \frac{\sqrt{1+\al}}{\al \sqrt{2\pi}}\,,
\eeq
where $I(\al) = (1+\al)\ln(1+\al) - \al$ for $\al \geq 0$. 
\end{claim}

\begin{proof}
Let $X_i \sim \po(1)-1$ for $i=1,2,\dots$ be the independent lattice type random variables. We have $\E(X_i)=0$ and $\textrm{Var}(X_i)=1$. Consider: (i) the moment generating function $M(s) := \E(e^{sX_i}) = \exp(e^{s}-s-1)$,   
(ii) the rate function $I(\al) = \sup_{s \in \R} \left(s\al - \ln M(z)\right) = (1+\al)\ln(1+\al) - \al$
which is attained at $s_\al=\ln(1+\al)$, and (iii) the variance $\sigma^2 := M''(s_\al)/M(s_\al) - \al^2=1+\al$.
Now the claim follows from Theorem~\ref{thm:br}.
\end{proof}

\end{document}